\documentclass[11pt,a4paper]{amsart}
\usepackage{hyperref,amsthm,amsfonts,amsmath,graphicx,float,graphicx}
\usepackage[numbers,sort&compress]{natbib}
\usepackage[lmargin=34mm,rmargin=34mm,tmargin=34mm,bmargin=34mm]{geometry}

\theoremstyle{plain}
\newtheorem{theorem}{Theorem}
\newtheorem{proposition}[theorem]{Proposition}
\newtheorem{conjecture}[theorem]{Conjecture}

\newcommand{\conjlabel}[1]{\label{con:#1}}
\newcommand{\conjref}[1]{Conjecture~\ref{con:#1}}    
    
\newcommand{\eqnlabel}[1]{\label{eqn:#1}}

\newcommand{\figlabel}[1]{\label{fig:#1}}
\newcommand{\figref}[1]{Figure~\ref{fig:#1}}
\newcommand{\seclabel}[1]{\label{sec:#1}}
\newcommand{\secref}[1]{Section~\ref{sec:#1}}
\newcommand{\proplabel}[1]{\label{prop:#1}}

\newcommand{\spacing}[1]{\renewcommand{\baselinestretch}{#1}\setlength{\footnotesep}{\baselinestretch\footnotesep}}
\spacing{1.2}

\newcommand{\V}{\mathcal{V}}

\newcommand{\Figure}[4][htb]{
\begin{figure}[#1]
	\vspace*{1ex}
	\begin{center}#3\end{center}
	\vspace*{-1ex}
	\caption{\figlabel{#2}#4}
\end{figure}
}

\newcommand{\half}{\ensuremath{\protect\tfrac{1}{2}}}
\newcommand{\ceil}[1]{\ensuremath{\protect\lceil#1\rceil}}

\newcommand{\floor}[1]{\ensuremath{\protect\lfloor#1\rfloor}}

\begin{document}

\title[On Visibility and Blockers]{On Visibility and Blockers}

\author[]{Attila P\'or}
\address{Department of Mathematics, 
  Western Kentucky University
\newline Bowling Green,  Kentucky, U.S.A.}
\email{attila.por@wku.edu}

\author[]{David~R.~Wood}
\address{Department of Mathematics and Statistics, 
 The University of Melbourne
\newline Melbourne, Australia}
\email{woodd@unimelb.edu.au}

\begin{abstract}
This expository paper discusses some conjectures related to visibility
and blockers for sets of points in the plane.
\end{abstract}


\maketitle



\section{Visibility Graphs}
\seclabel{VisGraphs}


Let $P$ be a finite set of points in the plane. Two distinct points
$v$ and $w$ in the plane are \emph{visible} with respect to $P$ if no
point in $P$ is in the open line segment $\overline{vw}$. The
\emph{visibility graph} $\V(P)$ of $P$ has vertex set $P$, where two
distinct points $v,w\in P$ are adjacent if and only if they are
visible with respect to $P$. So $\V(P)$ is obtained by drawing lines
through each pair of points in $P$, where two points are adjacent if they
are consecutive on a such a line. Visibility graphs have many
interesting properties. For example, they have diameter at most 2
(assuming $P$ is not collinear).  Consider the following
Ramsey-theoretic conjecture by \citet{KPW-DCG05}, which has recently
received considerable attention \citep{Luigi,Matousek09,EmptyPentagon}.


\begin{conjecture}[Big-Line-Big-Clique Conjecture \citep{KPW-DCG05}]
  \conjlabel{BigClique} For all positive integers $k$ and $\ell$ there is an
  integer $n$ such that for every finite set $P$ of at least $n$
  points in the plane:
  \begin{itemize}
  \item $P$ contains $\ell$ collinear points, or
  \item $P$ contains $k$ pairwise visible points (that is, $\V(P)$
    contains a $k$-clique).
  \end{itemize}
\end{conjecture}

\conjref{BigClique} is true for $k\leq 5$ or $\ell\leq3$
\citep{KPW-DCG05,Luigi,EmptyPentagon}, and is open for $k=6$ or
$\ell=4$.


Note that the natural approach for attacking the Big-Line-Big-Clique
Conjecture using extremal graph theory fails. \citet{Turan41} proved
that every $n$-vertex graph with more edges than the Tur\'an graph
$T_{n,k}$ contains $K_{k+1}$ as a subgraph\footnote{ Let $T_{n,k}$ be
  the $k$-coloured graph with $n_i $ vertices in the $i$-th colour
  class, where two vertices are adjacent  if and only if they have
  distinct colours, and $n=\sum_in_i$ and $|n_i-n_j|\leq
  1$ for all $i,j\in[k]$.}. Thus the Big-Line-Big-Clique Conjecture would be
proved if every sufficiently large visibility graph with no $\ell$
collinear points has more edges than $T_{n,k-1}$.  However,
\citet{BGS-GM74} and \citet{FurediPalasti-PAMS84} constructed sets $P$
of $n$ points with no four collinear, such that $P$ determines
$\frac{n^2}{6}-O(n)$ lines each containing three points. Thus $\V(P)$
has $\frac{n^2}{3}+O(n)$ edges, which is less than the number of edges
in $T_{n,k-1}$ for all $k\geq5$ and large $n$. These examples show
that the number of edges in a visibility graph with no four collinear
points is not enough to necessarily imply the existance of a large
clique via Tur\'an's Theorem.

Consider the following weakening of \conjref{BigClique}, due to Jan
Kara [private communication, 2005].

\begin{conjecture}
  \conjlabel{BigChromaticNumber} 
For all integers $k\geq2$ and $\ell\geq1$ there is an integer $n$ such that
if $P$ is a finite set of at least $n$ points in the plane, and each point in $P$ is assigned one of $k-1$ colours, then:
 \begin{itemize}
  \item $P$ contains $\ell$ collinear points, or
  \item some pair of visible points in $P$ receive the same colour\\
    (that is, the visibility graph $\V(P)$ has chromatic number
    $\chi(\V(P))\geq k$).
  \end{itemize}
\end{conjecture}

\conjref{BigClique} implies \conjref{BigChromaticNumber} since the
chromatic number of any graph containing a $k$-clique is at least
$k$. Thus \conjref{BigChromaticNumber} is true for $k\leq 5$ or
$\ell\leq3$. Consider a proper colouring of a visibility graph
$\V(P)$.  That is, visible points are coloured differently. 
In each colour class $C$, no two vertices are visible. So
the vertices not in $C$ `block' the lines of visibility amongst
vertices in $C$. This idea leads to the following definitions that
were independently introduced by \citet{Matousek09} amongst others.

A point $x$ in the plane \emph{blocks} two points $v$ and $w$ if
$x\in\overline{vw}$. Let $P$ be a finite set of points in the plane. A
set $B$ of points in the plane \emph{blocks} $P$ if $P\cap
B=\emptyset$ and for all distinct $v,w\in P$ there is a point in $B$
that blocks $v$ and $w$. That is, no two points in $P$ are visible
with respect to $P\cup B$, or alternatively, $P$ is an independent set
in $\V(P\cup B)$.

The purpose of this expository paper is to discuss some conjectures
related to blocking sets. We remark that in the last few years, a
number of researchers have started studying blocking sets around the
same time (see \citep{Matousek09,DPT09,Pinchasi} and the named
researchers therein). So we expect that some of the observations in
this paper have been independently discovered by others.

\section{The Blocking Conjecture}
\seclabel{Blockers}

If $P$ is a set of collinear points then there is a set of $|P|-1$
points that block $P$. At the other extreme, how small can a blocking
set be if $P$ is in general position (that is, no three points are
collinear)? Let $b(P)$ be the minimum size of a set of points that
block $P$. Let $b(n)$ be the minimum of $b(P)$, where $P$ is a set of $n$
points in general position in the plane. We conjecture that every set
of points in general position requires a super-linear number of
blockers. 

\begin{conjecture}
  \conjlabel{Blockers} $\frac{b(n)}{n}\rightarrow\infty$ as $n
  \rightarrow\infty$.
\end{conjecture}

In fact, \citet{Pinchasi} conjectured that $b(n)\in\Omega(n\log
n)$. Linear lower bounds on $b(n)$ are known \citep{Matousek09,DPT09}.
Let $P$ be a set of $n$ points in the plane in general position with
$t$ vertices on the boundary of the convex hull.  Each edge of a
triangulation of $P$ requires a distinct blocker, and every
triangulation of $P$ has $3n-3-t$ edges.  So every blocking set of $P$
has at least $3n-3-t\geq 2n-3$ vertices, and $b(n)\geq 2n-3$.
\citet{DPT09} improved this bound to $b(n)\geq(\frac{25}{8}-o(1))n$.

\section{Blocking Graph Drawings}
\seclabel{TopoBlockers}

A \emph{drawing} of a graph $G$ represents each vertex of $G$ by a
distinct point in the plane, and represents each edge of $G$ by a
simple closed curve between its endpoints, such that a vertex $v$
intersects an edge $e$ only if $v$ is an endpoint of $e$.  We do not
distinguish between graph elements and their representation in a
drawing. Note that multiple edges may intersect at a common point. A
drawing is \emph{simple} if any two edges intersect at most once, at a
common endpoint or as a proper crossing (``kissing'' edges are not
allowed). A drawing is \emph{geometric} if each edge is a straight
line-segment. Obviously, every geometric drawing is simple.

Blockers for point sets generalise for graph drawings as follows. A
set of points $B$ \emph{blocks} a drawing of a graph $G$ if no vertex
of $G$ is in $B$ and every edge of $G$ contains some point in $B$.
Observe that if $P$ is a set of points in general position, then $B$
blocks $P$ if and only if $B$ blocks the geometric drawing of the
complete graph with vertices drawn at $P$.

Some geometry is needed in \conjref{Blockers}, in the sense that $K_n$
has a simple (non-geometric) drawing that can be blocked by $2n-3$
blockers. As illustrated in \figref{BlockedTopoDrawing7}, if
$V(K_n)=\{v_1,\dots,v_n\}$ then place $v_i$ at $(i,0)$ and draw each
edge $v_iv_j$ with $i<j$ by a curve from $v_i$ into the upper
half-plane, through the point $(-i-j,0)$, into the lower half-plane,
and across to $v_j$. As illustrated in \figref{BlockedTopoDrawing7},
the edges can be drawn so that two edges intersect at most once. Each
edge is blocked by one of the $2n-3$ points in
$\{(-k,0):k\in[3,2n-1]\}$.  This observation improves upon a $O(n\log
n)$ upper bound on the number of blockers in a simple drawing of
$K_n$, due to \citet{DPT09}. A similar construction is due to
\citet{HM92}; see \citet{PST-DCG03}.  Note that at least $n-1$
blockers are needed for every simple drawing of $K_n$ (since each
point can block at most $\frac{n}{2}$ edges).

\begin{conjecture}
The minimum number of blockers in a simple drawing of
$K_n$ equals  $2n-3$.
\end{conjecture}

\Figure{BlockedTopoDrawing7}{\includegraphics[width=\textwidth]{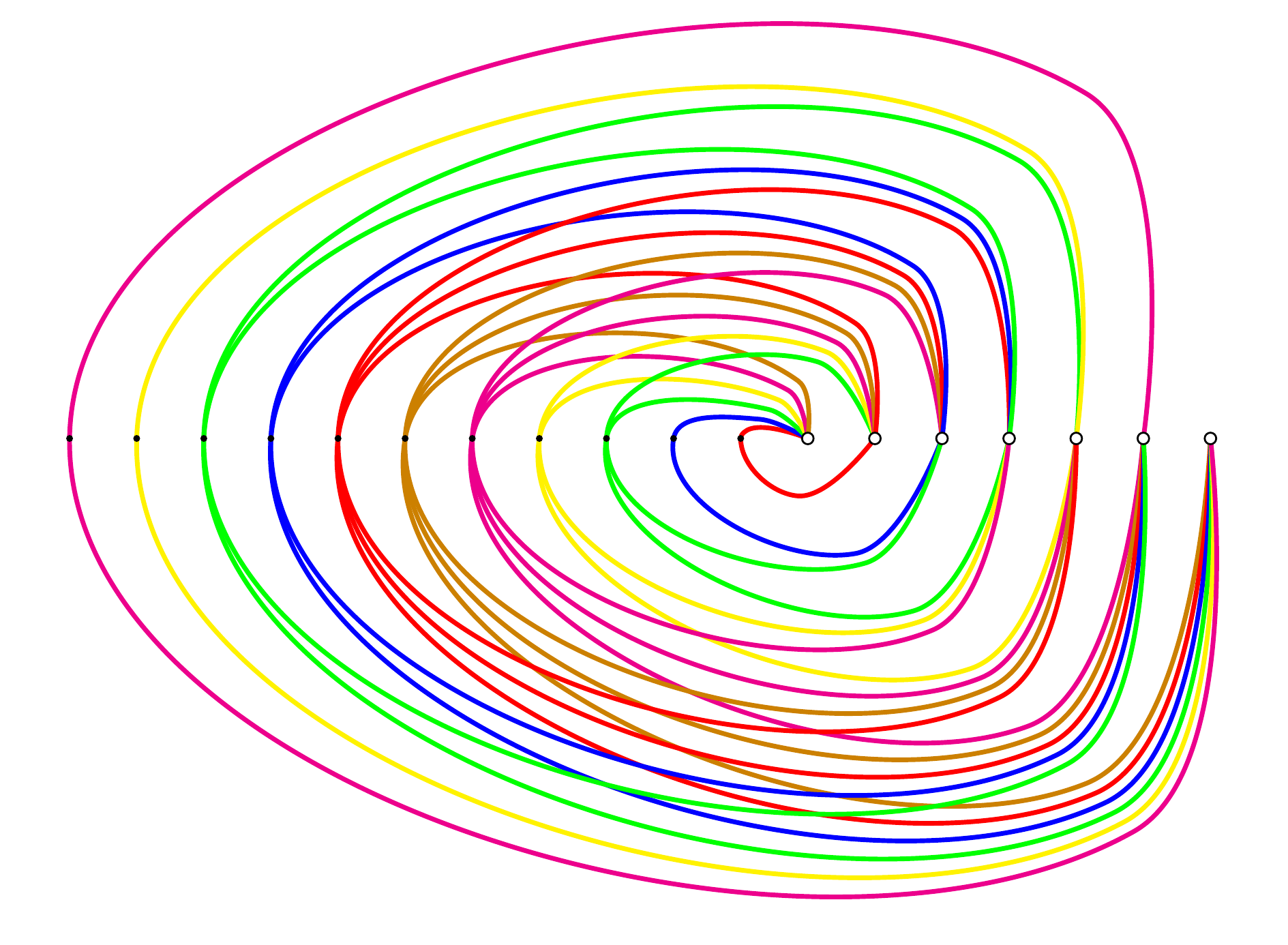}}{A drawing of $K_7$ blocked by $11$ blockers.}


While this example suggests that geometry is needed in
\conjref{Blockers}, Stefan Langerman [personal communication, 2009]
proposed an alternative. A drawing of a graph is \emph{extendable} if
the edges are contained in a pseudoline arrangment; that is, for each
edge $e$ there is a simple unbounded curve $C_e$ containing $e$, such
that for all distinct edges $e$ and $e'$, the curves $C_e$ and
$C_{e'}$ intersect at most once. Observe that the above simple drawing
that can be blocked by $O(n)$ blockers is not extendable. We
conjecture that every extendible simple drawing of $K_n$ needs a
super-linear number of blockers.

\section{Midpoints and Freiman's Theorem}
\seclabel{Midpoints}

\conjref{Blockers} is related to results by
\citet{Pach-Midpoints-Geom03} about midpoints. For a set $P$ of points
in the plane, let $m(P)$ be the number of midpoints determined by
distinct points in $P$; that is, $m(P):=|\{\half(x+y):x,y\in P,x\neq
y\}|$. Let $m(n)$ be the minimum of $m(P)$, where $P$ is a set of $n$
points in general position in the plane. Since midpoints are also
blockers, $b(n)\leq m(n)$.  \citet{Pach-Midpoints-Geom03} (and later
\citet{Matousek09}) constructed a set of $n$ points in general
position in the plane that determine at most $nc^{\sqrt{\log n}}$
midpoints for some contant $c$. Thus $$b(n)\leq m(n)\leq
nc^{\sqrt{\log n}}\enspace.$$ (This function is between $n\log n$ and
$n^{1+\epsilon}$.) Moreover, \citet{Pach-Midpoints-Geom03} proved that
$\frac{m(n)}{n}\rightarrow\infty$ as $n\rightarrow\infty$. Thus
\conjref{Blockers} would stregthen this lower bound on $m(n)$.

Pach's proof of this lower bound is based on Freiman's
Theorem\footnote{A $d$-dimensional \emph{progression} in the plane is
  a set $\{v_0+x_1v_1+\dots+x_dv_d:x_i\in[1,n_i]\}$ for
  some vectors $v_0,\dots,v_d\in\mathbb{R}^2$. Freiman's Theorem is
  usually stated in terms of the \emph{sum set}
  $P+P:=\{x+y:x,y\in P\}$. Clearly $m(P)\leq|P+P|\leq m(P)+|P|$.
  Freiman's Theorem actually applies in any abelian group; see
  \citep{TaoVu06}. See
  \citep{Stanchescu98,Stanchescu08,Stanchescu99,Fishburn-JNT90} for
  more on Freiman's Theorem in the plane.}, which implies that if
$m(P)=\alpha n$ for some set $P$ of $n$ points in the plane (not
necessarily in general position), then $P$ is a subset of a
$d$-dimensional progression of size at most $\beta n$, for some $d$
and $\beta$ depending only on $\alpha$. Pach concluded that at least
$\frac{1}{\beta}n^{1/d}$ points in $P$ are collinear. Since no three
points in $P$ are collinear, $n$ is bounded by a function of
$\alpha$. It follows that $\frac{m(n)}{n}\rightarrow\infty$.  We can
obtain a more precise lower bound on $m(n)$ as follows.
\citet{Chang-DMJ02} proved that in Freiman's Theorem one can take
$d=\floor{\alpha-1}$ and $\beta=2^{c \alpha^2\log^3\alpha}$, for some
absolute constant $c>0$. Applying this result in Pach's proof, it
follows that $\alpha^3\log^3\alpha \geq \frac{1}{c}\log n$. Hence for
all $\epsilon>0$ and sufficiently large $n$, for some absolute
constant $c>0$,
\begin{equation}
\eqnlabel{newm}
m(n)\geq c n(\log n)^{1/(3+\epsilon)}\enspace.
\end{equation}
Analogously, the following conjectured `convex combination' version of
Freiman's Theorem would establish \conjref{Blockers}.

\begin{conjecture}
  \conjlabel{GeneralFreiman} 
Let $P$ be a set of points in the plane
  with at most $\half|P|$ points collinear. Suppose that $P$ can be
  blocked by some set $B$ with $|B|\leq\alpha|P|$. That is, for all
  distinct $x,y\in P$ there is a real number $\gamma\in(0,1)$, such
  that $\gamma x+(1-\gamma)y\in B$. Then $P$ is a subset of a
  $d$-dimensional progression of size at most $\beta|P|$,  for some
  $d$ and $\beta$ depending only on $\alpha$.
\end{conjecture}

Note that some assumption on the number of collinear points is needed
in \conjref{GeneralFreiman}. For example, a set of $n$ random
collinear points can be blocked by $n-1$ points, but is not a subset
of a progression of bounded dimension and linear size. This conjecture
generalises Freiman's Theorem for the plane, which assumes
$\alpha=\frac{1}{2}$ for all $x,y\in P$.

While Freiman's Theorem applies in some sense for sum sets along the
edges of any dense graph \citep{ER-JCTA06}, it is worth noting that
there is a geometric drawing of $K_{n,n}$ that can be
blocked by $O(n)$ blockers. Say the colour classes of $K_{n,n}$ are
$\{v_1,\dots,v_n\}$ and $\{w_1,\dots,w_n\}$. Position $v_i$ at
$(2i,0)$, and $w_j$ at $(2j,2)$. Thus $v_iw_j$ is blocked by
$(i+j,1)$, and $\{(i,1):i\in[2,2n]\}$ is a set of $2n-1$ points
blocking every edge. In fact, there is a geometric drawing of
$K_{n,n}$ with its vertices in general position that can be similarly
blocked. Position $v_i$ at $(-2^i,2^{2i})$ and $w_j$ at
$(2^j,2^{2j})$. These points lie on opposite sides of the parabola
$y=x^2$. The edge $v_iw_j$ is blocked by $(0,2^{i+j})$, and
$\{(0,2^i):i\in[2,2n]\}$ is a set of $2n-1$ points blocking every
edge.

In general, say $S=\{s_1,\dots,s_n\}$ is a set of $n$ positive
integers. Draw $K_{n,n}$ by positioning each $v_i$ at $(-s_i,s_i^2)$
and each $w_j$ at $(s_j,s_j^2)$ (again on opposite sides of the
parabola $y=x^2$). Say we block every edge by a point on the
y-axis. The edge $v_iw_j$ crosses the y-axis at $(0,s_is_j)$.  Thus to
have few blockers, $S$ should be chosen so that the product set
$S\cdot S := \{ab :a,b \in S\}$ is small.  Geometric progessions, such
as $2^1,2^2,\dots,2^n$, minimise the size of the product set (leading
to the construction of $K_{n,n}$ above).  It is interesting that both
sum sets (that is, midpoints) and product sets appear to be related to
blocking sets. There is a known trade-off between the sizes of sum
sets and product sets (so-called \emph{sum-product estimates}). In
particular, $|S+S|$ or $|S\cdot S|$ is at least $c|S|^{1+\epsilon}$
for some $c>0$ and $\epsilon>0$; see
\citep{ES83,Solymosi05,Chang03,Chen-PAMS99,Elekes97}. Especially given
that geometric methods based on the Szemer\'edi-Trotter theorem can be
used to prove such a result \citep{Elekes97}, it is plausible that
sum-product estimates might shed some light on \conjref{Blockers}.

\section{Colouring Edges}


Now consider edge-colourings of graph drawings, such that if two edges
have the same colour, then they cross. This idea is related to
blockers, since if a graph drawing can be blocked by $b$ blockers,
then it can be coloured with $b$ colours. Let $t(n)$ be the minimum
integer such that the edges in some geometric drawing of $K_n$ can be
coloured with $t(n)$ colours such that every monochromatic pair of
edges cross. Each colour class is called a \emph{crossing family}
\citep{AEGKKPS94}. Hence $t(n)\leq b(n)$. We conjecture the following
strengthening of \conjref{Blockers}.

\begin{conjecture}
$\frac{t(n)}{n}\rightarrow\infty$ as $n\rightarrow\infty$. 
\end{conjecture}

The analogous conjecture could be made for extendible simple drawings
of $K_n$.

\section{Convex Position}

For point sets in convex position, the above edge-colouring problem is
equivalent to covering a circle graph by cliques. It follows from a
result by \citet{Kostochka88} (see \citep{KK-DM97}) that the minimum
number of colours is at least $n\ln n-c$ and at most $n\ln n+cn$, for
some constant $c$. Thus the number of blockers for a point set in convex
position is at least $n\log n-c$. We conjecture the answer is quadratic.

\begin{conjecture} Every set of $n$ points in convex position require
$\Omega(n^2)$ blockers.
\end{conjecture}

For $n$ equally spaced points around a circle, at least
$\frac{n^2}{14}-O(n)$ blockers are required, since except for the
point in the centre, at most $7$ edges intersect at a common interior
point \citep{PooRub-SJDM98}. This property does not hold for arbitrary
points in convex position, since as described in \secref{Midpoints},
for the point set $P=\{(-2^i,2^{2i}),(2^i,2^{2i}):i\in[1, n]\}$, the
point $(0,2^k)$ blocks each edge $(-2^i,2^{2i})(2^j,2^{2j})$ for which
$k=i+j$. Thus $\Omega(n)$ points on the y-axis each block $\Omega(n)$
edges.

Note that \citet{EFF-SJDM91} proved that the minimum number of
midpoints for a set of $n$ points in convex position is between
$0.8\binom{n}{2}$ and $0.9\binom{n}{2}$. 

\section{Point Sets with Bounded Collinearities}
\seclabel{BoundedCollinearities}

Now consider midpoints and blocking sets for point sets with
a bounded number of collinear points.  Let $m_\ell(n)$ be the minimum
number of midpoints determined by some set of $n$ points in the plane
with no $\ell$ collinear points. Thus $m_3(n)=m(n)$.  The proof of the
lower bound on $m(n)$ described in \secref{Midpoints} generalises to show that for all
$\epsilon>0$ and sufficiently large $n>n(\epsilon)$, for some absolute
constant $c$, $$m_\ell(n)\geq \frac{c}{\log\ell}\, n(\log n)^{1/(3+\epsilon)}\enspace.$$
Similarly, let $b_\ell(n)$ be the minimum integer such that every set
of $n$ points in the plane with no $\ell$ collinear points is blocked
by some set of $b_\ell(n)$ points. Thus $b_3(n)=b(n)$. We conjecture that
$b_\ell(n)$ is also super-linear in $n$ for fixed $\ell$.

\begin{conjecture}
  \conjlabel{GeneralBlockers} For all fixed $\ell$, we have
  $\frac{b_\ell(n)}{n}\rightarrow\infty$ as $n\rightarrow\infty$.
\end{conjecture}

\begin{proposition}
  \proplabel{GeneralBlockersImpliesBigChromaticNumber}
  \conjref{GeneralBlockers} implies \conjref{BigChromaticNumber}.
\end{proposition}

\begin{proof} 
  Suppose on the contrary that \conjref{GeneralBlockers} holds but
  \conjref{BigChromaticNumber} does not.  Thus there are constants
  $\ell$ and $k$, and there are arbitrarily large point sets $P$
  containing no $\ell$ collinear points, and with $\chi(\V(P))\leq k$.
  \conjref{GeneralBlockers} implies that $b_{\ell}(n)\geq n\cdot
  g_{\ell}(n)$ for some non-decreasing function $g_{\ell}$ for which
  $g_{\ell}(n)\rightarrow\infty$ as $n\rightarrow\infty$.  Thus there
  is an integer $n'$ such that $g_\ell(n')>k-1$.  Let $P$ be a set of
  $n\geq kn'$ points, containing no $\ell$ collinear points, and with
  $\chi(\V(P))\leq k$.  Let $S$ be the largest colour class in a
  $k$-colouring of $\V(P)$.  Thus $S$ has no $\ell$ collinear points
  and $P-S$ blocks $S$.  That is, there is a set of
  $s=\ceil{\frac{n}{k}}$ points blocked by a set of $n-s$ points.
  Thus $b_{\ell}(s)\leq n-s\leq n(1-\frac{1}{k})$. On the other hand,
  $b_{\ell}(s)\geq s\cdot g_{\ell}(s)\geq\frac{n}{k}\cdot
  g_{\ell}(s)$.  Hence $\frac{n}{k}\cdot g_{\ell}(s)\leq
  n(1-\frac{1}{k})$ and $g_{\ell}(s)\leq k-1$.  Since $n'\leq s$ and
  $g$ is non-decreasing, $g_{\ell}(n')\leq k-1$, which is the desired
  contradiction.
\end{proof}

\section{A Final Conjecture}

We finish the paper with a strengthening of
\conjref{BigChromaticNumber}. 

\begin{conjecture}
\conjlabel{MonoLines}
For all positive integers $k$ and $\ell$ there is an integer $n$ such that if $P$ is a set of at least $n$ points in the plane, and each point in $P$ is assigned one of $k$ colours, then:
\begin{itemize}
\item $P$ contains $\ell$ collinear points, or
\item $P$ contains  a monochromatic line (that is, a maximal set of collinear points, all receiving the same colour). 
\end{itemize}
\end{conjecture}

\conjref{MonoLines} is trivially true for $k=1$ and $n=2$, or $\ell=3$
and $n=k+1$. The Motzkin-Rabin Theorem says that it is true for $k=2$
with $n=\ell$; see
\citep{Motzkin67,PS-AMM04,BorMos90}. \conjref{MonoLines} is related to
the Hales-Jewett Theorem \citep{Shelah-JAMS88,DHJ,HalesJewett-TAMS63},
which states that for sufficiently large $d$, every $k$-colouring of
the grid $[1,\ell-1]^d$ contains a monochromatic ``combinatorial'' line
of length $\ell-1$.


\section*{Acknowledgements} This research was initiated in 2004 at the
Department of Applied Mathematics of Charles University in
Prague. Thanks to Jaroslav Ne{\v{s}}et{\v{r}}il, Jan Kratochv{\'\i}l
and Pavel Valtr for their generous hospitality. The second author
thanks the numerous people with whom he has had stimulating
discussions regarding the contents of this paper.


\begin{thebibliography}{35}
\providecommand{\natexlab}[1]{#1}
\providecommand{\url}[1]{\texttt{#1}}
\providecommand{\urlprefix}{}
\expandafter\ifx\csname urlstyle\endcsname\relax
  \providecommand{\doi}[1]{doi:\discretionary{}{}{}#1}\else
  \providecommand{\doi}{doi:\discretionary{}{}{}\begingroup
  \urlstyle{rm}\Url}\fi

\bibitem[{Abel et~al.(2009)Abel, Ballinger, Bose, Collette, Dujmovi\'{c},
  Hurtado, Kominers, Langerman, P\'or, and Wood}]{EmptyPentagon}
\textsc{Zachary Abel, Brad Ballinger, Prosenjit Bose, S\'ebastien Collette,
  Vida Dujmovi\'{c}, Ferran Hurtado, Scott~D. Kominers, Stefan Langerman,
  Attila P\'or, and David~R. Wood}.
\newblock Every large point set contains many collinear points or an empty
  pentagon.
\newblock In \emph{Proc. 21st Canadian Conference on Computational Geometry
  (CCCG '09)}, pp. 99--102. 2009.
\newblock \urlprefix\url{http://arxiv.org/abs/0904.0262}.

\bibitem[{Addario-Berry et~al.(2007)Addario-Berry, Fernandes, Kohayakawa,
  de~Pina, and Wakabayashi}]{Luigi}
\textsc{Louigi Addario-Berry, Cristina Fernandes, Yoshiharu Kohayakawa,
  Jos~Coelho de~Pina, and Yoshiko Wakabayashi}.
\newblock On a geometric {R}amsey-style problem, 2007.
\newblock \urlprefix\url{http://crm.umontreal.ca/cal/en/mois200708.html}.

\bibitem[{Aronov et~al.(1994)Aronov, Erd{\H{o}}s, Goddard, Kleitman, Klugerman,
  Pach, and Schulman}]{AEGKKPS94}
\textsc{Boris Aronov, Paul Erd{\H{o}}s, Wayne Goddard, Daniel~J. Kleitman,
  Michael Klugerman, J\'{a}nos Pach, and Leonard~J. Schulman}.
\newblock Crossing families.
\newblock \emph{Combinatorica}, 14(2):127--134, 1994.
\newblock \urlprefix\url{http://dx.doi.org/10.1007/BF01215345}.

\bibitem[{Borwein and Moser(1990)}]{BorMos90}
\textsc{Peter Borwein and William O.~J. Moser}.
\newblock A survey of {S}ylvester's problem and its generalizations.
\newblock \emph{Aequationes Math.}, 40(2-3):111--135, 1990.

\bibitem[{Burr et~al.(1974)Burr, Gr{\"u}nbaum, and Sloane}]{BGS-GM74}
\textsc{Stefan~A. Burr, Branko Gr{\"u}nbaum, and Neil J.~A. Sloane}.
\newblock The orchard problem.
\newblock \emph{Geometriae Dedicata}, 2:397--424, 1974.
\newblock \urlprefix\url{http://dx.doi.org/10.1007/BF00147569}.

\bibitem[{Chang(2002)}]{Chang-DMJ02}
\textsc{Mei-Chu Chang}.
\newblock A polynomial bound in {F}reiman's theorem.
\newblock \emph{Duke Math. J.}, 113(3):399--419, 2002.
\newblock \urlprefix\url{http://dx.doi.org/10.1215/S0012-7094-02-11331-3}.

\bibitem[{Chang(2003)}]{Chang03}
\textsc{Mei-Chu Chang}.
\newblock Factorization in generalized arithmetic progressions and applications
  to the {E}rd{\H o}s-{S}zemer\'edi sum-product problems.
\newblock \emph{Geom. Funct. Anal.}, 13(4):720--736, 2003.
\newblock \urlprefix\url{http://dx.doi.org/10.1007/s00039-003-0428-5}.

\bibitem[{Chen(1999)}]{Chen-PAMS99}
\textsc{Yong-Gao Chen}.
\newblock On sums and products of integers.
\newblock \emph{Proc. Amer. Math. Soc.}, 127(7):1927--1933, 1999.
\newblock \urlprefix\url{http://dx.doi.org/10.1090/S0002-9939-99-04833-9}.

\bibitem[{Dumitrescu et~al.(2009)Dumitrescu, Pach, and T{\'o}th}]{DPT09}
\textsc{Adrian Dumitrescu, J{\'a}nos Pach, and G{\'e}za T{\'o}th}.
\newblock A note on blocking visibility between points.
\newblock \emph{Geombinatorics}, 19(1):67--73, 2009.
\newblock \urlprefix\url{http://www.cs.uwm.edu/faculty/ad/blocking.pdf}.

\bibitem[{Elekes(1997)}]{Elekes97}
\textsc{Gy{\"o}rgy Elekes}.
\newblock On the number of sums and products.
\newblock \emph{Acta Arith.}, 81(4):365--367, 1997.

\bibitem[{Elekes and Ruzsa(2006)}]{ER-JCTA06}
\textsc{Gy{\"o}rgy Elekes and Imre~Z. Ruzsa}.
\newblock The structure of sets with few sums along a graph.
\newblock \emph{J. Combin. Theory Ser. A}, 113(7):1476--1500, 2006.
\newblock \urlprefix\url{http://dx.doi.org/10.1016/j.jcta.2005.10.011}.

\bibitem[{Erd{\H{o}}s et~al.(1991)Erd{\H{o}}s, Fishburn, and
  F{\"u}redi}]{EFF-SJDM91}
\textsc{Paul Erd{\H{o}}s, Peter Fishburn, and Zolt{\'a}n F{\"u}redi}.
\newblock Midpoints of diagonals of convex {$n$}-gons.
\newblock \emph{SIAM J. Discrete Math.}, 4(3):329--341, 1991.
\newblock \urlprefix\url{http://dx.doi.org/10.1137/0404030}.

\bibitem[{Erd{\H{o}}s and Szemer{\'e}di(1983)}]{ES83}
\textsc{Paul Erd{\H{o}}s and Endre Szemer{\'e}di}.
\newblock On sums and products of integers.
\newblock In \emph{Studies in pure mathematics}, pp. 213--218. Birkh\"auser,
  Basel, 1983.

\bibitem[{Fishburn(1990)}]{Fishburn-JNT90}
\textsc{Peter~C. Fishburn}.
\newblock On a contribution of {F}reiman to additive number theory.
\newblock \emph{J. Number Theory}, 35(3):325--334, 1990.
\newblock \urlprefix\url{http://dx.doi.org/10.1016/0022-314X(90)90120-G}.

\bibitem[{F{\"u}redi and Pal{\'a}sti(1984)}]{FurediPalasti-PAMS84}
\textsc{Zolt{\'a}n F{\"u}redi and Ilona Pal{\'a}sti}.
\newblock Arrangements of lines with a large number of triangles.
\newblock \emph{Proc. Amer. Math. Soc.}, 92(4):561--566, 1984.
\newblock \urlprefix\url{http://dx.doi.org/10.2307/2045427}.

\bibitem[{Hales and Jewett(1963)}]{HalesJewett-TAMS63}
\textsc{Alfred~W. Hales and Robert~I. Jewett}.
\newblock Regularity and positional games.
\newblock \emph{Trans. Amer. Math. Soc.}, 106:222--229, 1963.

\bibitem[{Harborth and Mengersen(1992)}]{HM92}
\textsc{Heiko Harborth and Ingrid Mengersen}.
\newblock Drawings of the complete graph with maximum number of crossings.
\newblock In \emph{Proc. 23rd {S}outheastern {I}nternational {C}onf. on
  {C}ombinatorics, {G}raph {T}heory, and {C}omputing (1992)}, vol.~88 of
  \emph{Congr. Numer.}, pp. 225--228. 1992.

\bibitem[{K{\'a}ra et~al.(2005)K{\'a}ra, P{\'o}r, and Wood}]{KPW-DCG05}
\textsc{Jan K{\'a}ra, Attila P{\'o}r, and David~R. Wood}.
\newblock On the chromatic number of the visibility graph of a set of points in
  the plane.
\newblock \emph{Discrete Comput. Geom.}, 34(3):497--506, 2005.
\newblock \urlprefix\url{http://dx.doi.org/10.1007/s00454-005-1177-z}.

\bibitem[{Kostochka and Kratochv{\'{\i}}l(1997)}]{KK-DM97}
\textsc{Alexandr Kostochka and Jan Kratochv{\'{\i}}l}.
\newblock Covering and coloring polygon-circle graphs.
\newblock \emph{Discrete Math.}, 163(1--3):299--305, 1997.
\newblock \urlprefix\url{http://dx.doi.org/10.1016/S0012-365X(96)00344-5}.

\bibitem[{Kostochka(1988)}]{Kostochka88}
\textsc{Alexandr~V. Kostochka}.
\newblock Upper bounds on the chromatic number of graphs.
\newblock \emph{Trudy Inst. Mat. (Novosibirsk)}, 10:204--226, 1988.

\bibitem[{Matou{\v{s}}ek(2009)}]{Matousek09}
\textsc{Ji{\v{r}}{\'i} Matou{\v{s}}ek}.
\newblock Blocking visibility for points in general position.
\newblock \emph{Discrete Comput. Geom.}, 42(2):219--223, 2009.
\newblock \urlprefix\url{http://dx.doi.org/10.1007/s00454-009-9185-z}.

\bibitem[{Motzkin(1967)}]{Motzkin67}
\textsc{Theodore~S. Motzkin}.
\newblock Nonmixed connecting lines.
\newblock \emph{Notices Amer. Math. Soc.}, 14, 1967.
\newblock Abstract 67T-605.

\bibitem[{Pach(2003)}]{Pach-Midpoints-Geom03}
\textsc{J{\'a}nos Pach}.
\newblock Midpoints of segments induced by a point set.
\newblock \emph{Geombinatorics}, 13(2):98--105, 2003.

\bibitem[{Pach et~al.(2003)Pach, Solymosi, and T{\'o}th}]{PST-DCG03}
\textsc{J{\'a}nos Pach, J{\'o}zsef Solymosi, and G{\'e}za T{\'o}th}.
\newblock Unavoidable configurations in complete topological graphs.
\newblock \emph{Discrete Comput. Geom.}, 30(2):311--320, 2003.
\newblock \urlprefix\url{http://dx.doi.org/10.1007/s00454-003-0012-9}.

\bibitem[{Pinchasi(2009)}]{Pinchasi}
\textsc{Rom Pinchasi}.
\newblock On some unrelated problems about planar arrangements of lines. {I}n
  \emph{Workshop II: Combinatorial Geometry. Combinatorics: Methods and
  Applications in Mathematics and Computer Science.} {Institute for Pure and
  Applied Mathematics, UCLA}, 2009.
\newblock \urlprefix\url{http://11011110.livejournal.com/184816.html}.

\bibitem[{{Polymath}(2009)}]{DHJ}
\textsc{D.~H.~J. {Polymath}}.
\newblock {A new proof of the density Hales-Jewett theorem}, 2009.
\newblock \urlprefix\url{http://arxiv.org/abs/0910.3926}.

\bibitem[{Poonen and Rubinstein(1998)}]{PooRub-SJDM98}
\textsc{Bjorn Poonen and Michael Rubinstein}.
\newblock The number of intersection points made by the diagonals of a regular
  polygon.
\newblock \emph{SIAM J. Discrete Math.}, 11(1):135--156, 1998.
\newblock \urlprefix\url{http://dx.doi.org/10.1137/S0895480195281246}.

\bibitem[{Pretorius and Swanepoel(2004)}]{PS-AMM04}
\textsc{Lou~M. Pretorius and Konrad~J. Swanepoel}.
\newblock An algorithmic proof of the {M}otzkin-{R}abin theorem on monochrome
  lines.
\newblock \emph{Amer. Math. Monthly}, 111(3):245--251, 2004.

\bibitem[{Shelah(1988)}]{Shelah-JAMS88}
\textsc{Saharon Shelah}.
\newblock Primitive recursive bounds for van der {W}aerden numbers.
\newblock \emph{J. Amer. Math. Soc.}, 1(3):683--697, 1988.

\bibitem[{Solymosi(2005)}]{Solymosi05}
\textsc{J{\'o}zsef Solymosi}.
\newblock On the number of sums and products.
\newblock \emph{Bull. London Math. Soc.}, 37(4):491--494, 2005.
\newblock \urlprefix\url{http://dx.doi.org/10.1112/S0024609305004261}.

\bibitem[{Stanchescu(1998)}]{Stanchescu98}
\textsc{Yonutz Stanchescu}.
\newblock On the structure of sets with small doubling property on the plane.
  {I}.
\newblock \emph{Acta Arith.}, 83(2):127--141, 1998.

\bibitem[{Stanchescu(1999)}]{Stanchescu99}
\textsc{Yonutz~V. Stanchescu}.
\newblock On the structure of sets of lattice points in the plane with a small
  doubling property.
\newblock \emph{Ast\'erisque}, (258):217--240, 1999.

\bibitem[{Stanchescu(2008)}]{Stanchescu08}
\textsc{Yonutz~V. Stanchescu}.
\newblock On the structure of sets with small doubling property on the plane.
  {II}.
\newblock \emph{Integers}, 8(2):A10, 20, 2008.
\newblock \urlprefix\url{http://www.integers-ejcnt.org/vol8-2.html}.

\bibitem[{Tao and Vu(2006)}]{TaoVu06}
\textsc{Terence Tao and Van Vu}.
\newblock \emph{Additive combinatorics}, vol. 105 of \emph{Cambridge Studies in
  Advanced Mathematics}.
\newblock Cambridge University Press, 2006.
\newblock \urlprefix\url{http://dx.doi.org/10.2277/0521853869}.

\bibitem[{Tur\'{a}n(1941)}]{Turan41}
\textsc{Paul Tur\'{a}n}.
\newblock On an extremal problem in graph theory.
\newblock \emph{Mat. Fiz. Lapok}, 48:436--452, 1941.

\end{thebibliography}

\def\cprime{$'$} \def\soft#1{\leavevmode\setbox0=\hbox{h}\dimen7=\ht0\advance
  \dimen7 by-1ex\relax\if t#1\relax\rlap{\raise.6\dimen7
  \hbox{\kern.3ex\char'47}}#1\relax\else\if T#1\relax
  \rlap{\raise.5\dimen7\hbox{\kern1.3ex\char'47}}#1\relax \else\if
  d#1\relax\rlap{\raise.5\dimen7\hbox{\kern.9ex \char'47}}#1\relax\else\if
  D#1\relax\rlap{\raise.5\dimen7 \hbox{\kern1.4ex\char'47}}#1\relax\else\if
  l#1\relax \rlap{\raise.5\dimen7\hbox{\kern.4ex\char'47}}#1\relax \else\if
  L#1\relax\rlap{\raise.5\dimen7\hbox{\kern.7ex
  \char'47}}#1\relax\else\message{accent \string\soft \space #1 not
  defined!}#1\relax\fi\fi\fi\fi\fi\fi} \def\Dbar{\leavevmode\lower.6ex\hbox to
  0pt{\hskip-.23ex\accent"16\hss}D}

\end{document}